\documentclass[11pt]{article}
\usepackage{amscd}
\usepackage{amsmath}
\usepackage{latexsym}
\usepackage{amsfonts}
\usepackage{amssymb}
\usepackage{amsthm}
\usepackage{graphicx}
\usepackage{verbatim}
\usepackage{enumerate}
\usepackage{xcolor}

\theoremstyle{plain}\newtheorem{definition}{Definition}[section]
\theoremstyle{definition}\newtheorem{theorem}{Theorem}[section]
\theoremstyle{plain}\newtheorem{lemma}[theorem]{Lemma}
\theoremstyle{plain}
\theoremstyle{plain}\newtheorem{proposition}[theorem]{Proposition}
\theoremstyle{remark}\newtheorem{remark}{Remark}[section]

\newcommand{\norm}[1]{\left\|#1\right\|}
\newcommand{\Div}{\mathrm{div}}

\newcommand{\esssup} {\mathop{\rm ess\,sup}}
\newcommand{\be}{\begin{equation}}
\newcommand{\ee}{\end{equation}}
 \newcommand{\ba}{\begin{aligned}}
 \newcommand{\ea}{\end{aligned}}
  \newcommand{\f}{\frac}
    
  \newcommand{\ben}{\begin{enumerate}}
   \newcommand{\een}{\end{enumerate}}
\newcommand{\te}{\text}

\newcommand{\Rmnum}[1]{\expandafter\@slowromancap\romannumeral #1@}

\begin{document}
\title{On
  possible time singular  points and eventual regularity of  weak solutions  to  the  fractional
 Navier-Stokes equations }
\date{}
\author{{\small \textbf{Quansen Jiu},$^{1}$\thanks{E-mail: jiuqs@mail.cnu.edu.cn}\quad
  \textbf{Yanqing Wang}$^{1}$\thanks{ E-mail: wangyanqing20056@gmail.com.} }}
 \maketitle {\small $^1$
School of Mathematical Sciences, Capital Normal University, Beijing
100048, P. R. China}

\begin{abstract}
In this paper,   we intend to reveal  how  the
fractional dissipation
  $(-\Delta)^{\alpha}$ affects
the regularity of
weak solutions to the 3d generalized  Navier-Stokes equations. Precisely, it will be shown that  the $(5-4\alpha)/2\alpha$ dimensional Hausdorff measure of
  possible time singular  points of  weak solutions on the interval $(0,\infty)$
is zero
when $5/6\le\alpha< 5/4$. To  this end, the eventual regularity for the weak solutions is firstly established in the  same range of $\alpha$.
It is worth noting that when the dissipation index $\alpha$ varies from $5/6$ to  $ 5/4$, the corresponding Hausdorff  dimension  is from $1$ to $0$.  Hence, it seems that the
   Hausdorff  dimension obtained is optimal.
   Our results rely on  the fact that the space  $H^{\alpha}$  is the critical space or subcritical space to this system when $\alpha\geq5/6$.


    \end{abstract}

\noindent {\bf MSC(2000):}\quad     35Q30,   35A02, 76D03. \\
{\bf Keywords:} Navier-Stokes equations,  fractional dissipation,
weak solutions, Hausdorff dimension,  eventual regularity.

\section{Introduction}
\setcounter{section}{1}\setcounter{equation}{0}

The incompressible time-dependent Navier-Stokes equations can be written as
\begin{align}
&u_{t}-\nu\Delta u+u\cdot \nabla u+\nabla p=0,~~\Div\,u=0,~(x,t)\in\Omega\times(0,\,T),\label{NS}
\end{align}
where the domain $\Omega$ denotes the whole space
$\mathbb{R}^{n}$ or a torus $\mathbb{T}^{n}$ with $n\geq2$.
The unknown vector field $u$ describes  velocity  of the flow, the
scalar function $p$ stands for the pressure of the fluid and  the positive constant $\nu $ is the
 viscosity coefficient.  We supplement problem \eqref{NS} with the divergence-free initial data $u(x,0)$.

It is well-known that the
Navier-Stokes system enjoys two fundamental properties. One is energy equality (inequality)
\begin{equation}\label{e:leray1}
\ba
E(u)=\esssup_{t}\int_{\mathbb{R}^{n}} |u (x,t)|^2\, dx+2\int^{\infty}_{0}  \int_{\mathbb{R}^{n}}
&|\nabla u(x,\tau)|^2 \,dx d\tau
\leq\int_{\mathbb{R}^{n}} |u(x,0)|^2 \,dx, \\
\ea
\end{equation}
for smooth solution (weak solutions).

The other one is the scaling transformation, namely, if the couple $\big(u(x,t), p(x,t)\big)$  solves problem \eqref{NS}, then so does $\big(u_\lambda(x,t), p_\lambda(x,t)\big)$ with
\be
\label{scaling}u_{\lambda}=\lambda u(\lambda x,\lambda^{2}t)\quad\text{and}\quad p_{\lambda}=\lambda ^{2}p(\lambda x,\lambda^{2}t).
\ee
This induces us to investigate problem \eqref{NS} in the critical spaces whose
 norm is invariant under scaling \eqref{scaling}. Thus, a natural candidate  is homogenous Sobolev space $\dot{H}^{(n-1)/2}$ or Lebesgue space $L^{n}$. A Bananch  space $X$ is said to be a supercritical space to \eqref{NS} if $\|u_{\lambda}\|_{X}\rightarrow\infty$  as $\lambda$ tends to $0$. The rest  are called as the subcritical spaces.
 We easily find that the energy space $L^{2}$ is a   supercritical space
 to system \eqref{NS}
  when the  spatial dimension is strictly greater than two.
Formally, based on \eqref{scaling}, we see that
\be\label{dimensional analysis}
E(u_{\lambda})=\esssup_{\lambda^{2}t}
 \int_{\mathbb{R}^{n}}\lambda^{2}|u(\lambda x,\, \lambda^{2}t)|^{2}\,dx
+2\int^{\infty}_{0}\int_{\mathbb{R}^{n}}\lambda^{4}|\nabla  u(\lambda x,\lambda^{2}\tau)|^{2}\,dxd\tau\\
=\lambda^{2-n}E(u),
\ee
which implies that $ E(u_{\lambda})\rightarrow\infty $ as
$\lambda\rightarrow0$.
In this sense, we may say that the 2d Navier-Stokes equations is critical and the Navier-Stokes equations is  supercritical when
the  spatial dimension is
greater than or equal to three.

 The global existence of weak solutions was successfully   proved by Leray on the Cauchy problem and by Hopf for the Dirichlet problem. Both of them made full use of
 the energy inequality \eqref{e:leray1}. Leray also proved that the weak solution of the 2d Navier-Stokes equations is regular  in \cite{[Leray]}.
Although there are   extensive studies  on  the  regularity  of  weak solutions to the 3d   Navier-Stokes equations
(See, e.g.,  \cite{[CKN],[ESS],[Galdi],[H],[Scheffer],[Scheffer2],[Serrin2],[Teman],[Teman2],[WW]}), the  regularity  of  weak solutions  is partially known until now. Particularly,
  Leray \cite{[Leray]} proved
 that   one dimension Lebesgue measure of the set of the possible time irregular  points for  weak solutions to the 3d
Navier-Stokes equations is zero. Moreover, Leray showed that every weak solution   becomes smooth after a large time, namely, eventual regularity of the
weak solutions. In \cite{[Scheffer]},
Scheffer  improved    Leray's upper bound of the
Hausdorff dimension of the possible time singular set of  weak solutions
 to $1/2$  (see also \cite[Section 6]{[Galdi]} and \cite[Chparter 5]{[Teman]}).
In this direction,
Scheffer \cite{[Scheffer2]} also
investigated the Hausdorff dimension of the space time
singular set of  weak solutions satisfying the local energy
inequality, which is the so-called partial regularity of suitable weak solutions.  The
well-known CKN theorem \cite{[CKN]} showed that one dimensional
Hausdorff measure of the possible
space-time singular points of  suitable weak solutions
 to the 3d Navier-Stokes equations  is zero. For more general results on partial regularity  of suitable weak solutions to  the non-stationary and stationary Navier-Stokes system, we refer the reader to the paper \cite{[WW]} and  the list of quotations there.

 Generally speaking,
it seems that there has not existed effective approach to deal with the supercritical equations so far.
Concomitantly, an interesting question is:
 Could one modify the
3d supercritical Navier-Stokes equations to become a critical (more regular)? As a matter of fact, this had been done   by Lions  in 1960s.
More precisely, Lions \cite{[Lions],[Lions2]} introduced the following equations involving fractional dissipation
 \be\label{GNS}
\left\{\ba
&u_{t}+\nu\Lambda^{2\alpha}u+u\cdot\nabla u+\nabla p=0,~~\te{div}~ u=0,\\
&u(0)=u(0,x).\ea\right.
\ee
The operator $\Lambda^{2\alpha}=(-\Delta)^{\alpha}$ is defined by
$\widehat{\Lambda^{2\alpha} f}(\xi)=|\xi|^{2\alpha}\hat{f}(\xi)$
in the  whole space,   where $\hat{f}(\xi)=\frac{1}{(2\pi)^{n}}\int_{\mathbb{R}^{n}}f (x)e^{-i\xi\cdot x}\,dx$; and by
$\widehat{\Lambda^{2\alpha} f}(k)=|\xi|^{2\alpha}\hat{f}(k),~k\in \mathbb{Z}^{n} $ on the torus, where $
\hat{f}(k)=\frac{1}{(2\pi)^{n}}\int_{\mathbb{T}^{n}}f (x)e^{-ik\cdot x}\,dx$. Here,
$\alpha\geq 0$ is  said to be the dissipation index.
Exactly as the Navier-Stokes equations, note that  if $u(x,t)$ solves \eqref{GNS},
$
u_{\lambda}=\lambda^{2\alpha-1}u(\lambda
x,\lambda^{2\alpha}t)
  $ is also a solution of \eqref{GNS} for any $\lambda\in \mathbb{R}^{+}$. Formally,
the corresponding energy is
\be\label{dimensional analysis2}\ba
E_{\alpha}(u_{\lambda})&=\esssup_{\lambda^{2\alpha}t}
\int_{\mathbb{R}^{n}}\lambda^{4\alpha-2}|u(\lambda
x,\lambda^{2\alpha}t) |^{2}\, dx+2\int^{\infty}_{0}\int_{\mathbb{R}^{n}}\lambda^{6\alpha-2}|\Lambda^{\alpha}u(\lambda
x,\lambda^{2\alpha}\tau)|^{2}\,dxd\tau\\
&=\lambda^{4\alpha-(n+2)}\Big\{
\esssup_{t}\int_{\mathbb{R}^{n}}|u(x,t)|^{2}\, dx+2\int^{\infty}_{0}\int_{\mathbb{R}^{n}}|\Lambda^{\alpha}u(x,\tau)|^{2}\,dxd\tau\Big\}\\
&=\lambda^{4\alpha-(n+2)}E_{\alpha}(u).
\ea\ee
It follows that $ E_{\alpha}(u_{\lambda})\rightarrow\infty $ as
$\lambda\rightarrow0$ when $\alpha<(n+2)/4 $. Just as the above, we say
that equations \eqref{GNS} is supercritical  if $\alpha<(n+2)/4$, critical
for $\alpha=(n+2)/4$
 and subcritical with   $\alpha>(n+2)/4$ .
Lions obtained  the global regular solution to the critical or subcritical    equations  \eqref{GNS} in \cite{[Lions2]}.

When $1<\alpha<5/4$, the three-dimensioanl fractional Navier-Stokes system is still supercritical equations and their global well-posedness theories
remain  open.   As pointed out by Katz and Pavlovi\'c in \cite{[KP]}, any improvement in the exponent $5/4$ could be viewed as genuine progress.
It is remarkable that
 the Navier-Stokes
equations with  fractional    dissipation have been studied from a mathematical viewpoint and some  interesting results
have been  obtained in \cite{[CMW],[KP],[MYZ],[Tao],[Wu1],[Wu3],[Wu2],[Wu4]}
and the references therein.
The existence and uniqueness of solutions to the generalized Navier-Stokes equations \eqref{GNS} in Besov spaces were established by Wu in \cite{[Wu3],[Wu2]}.
Recently,  Tao \cite{[Tao]} showed the global regularity for a logarithmically supercritical hyperdissipative Navier-Stokes system (see also \cite{[Wu4]}).
In  \cite{[KP]}, it was shown that  the Hausdorff dimension of the space
singular set at time of first blow up for smooth solution to equations \eqref{GNS}
is at most $5-4\alpha$ as $1<\alpha<5/4$.

In this paper, one target of our work is to address the question:
what is the  precise  effect of the   fractional  dissipation
to the upper bound on the Hausdorff dimension of the potential time singular set for the weak solutions to \eqref{GNS}.
We shall prove that  the $(5-4\alpha)/2\alpha$ dimensional Hausdorff measure of
  possible time singular  points of  weak solutions on the interval $(0,\infty)$
is zero in the case when $5/6\leq\alpha<5/4$.
This result not only is  an interpolation  between the  Scheffer's (Leray's)  Hausdorff dimension of the possible time singular set of weak solutions to the Navier-Stokes equations ($\alpha=1$)  and Lions's global  solvability
for the hyper-dissipative case  ($\alpha\geq5/4$), but also further  generalizes  Scheffer's (Leray's) classic work in the sense that $5/6\leq\alpha<1$.
 This   gives   an accurate  relationship between the
Hausdorff dimension of the potential
 time singular set of weak solutions and the
dissipation index  $\alpha$, which provides a perspective of the regularity of weak solutions to the fractional
 Navier-Stokes equations.

The other goal is to study the eventual regularity of  weak solutions to \eqref{GNS}.
More precisely, with  Leray's eventual regularization of weak solutions to the 3d Navier-Stokes equations in mind, we want  to explore how far
we may go beyond the standard   Laplacian dissipation $-\Delta$    and still prove the  eventual regularity of  weak solutions to the fractionally-dissipative Navier-Stokes equations.
In particular,
we can show   the eventual regularity of  weak solutions to \eqref{GNS} for $5/6\leq\alpha$.  It is natural to ask whether this result  is  true for $\alpha<5/6$.
This is also partially   motivated by the recently related work    for the supercritical quasi-geostrophic type equation by M. Dabkowski \cite{[Dabkowski]}, Miao and Xue \cite{[MX]} and   L. Silvestre \cite{[Silvestre]}.

In what follows, we study \eqref{GNS} in a periodic domain $\mathbb{T}^{3}=\mathbb{R}^{3}/\mathbb{Z}^{3}$.  The reason why we consider the periodic case  is that it is not obvious whether the strong energy inequality of weak solutions to \eqref{GNS} is valid in whole space when  $\alpha\neq1$.	
Utilizing
 integration by parts, the spatial periodicity of the solution and
the divergence free condition, we can derive that the solution $u$ of \eqref{GNS} satisfies $\frac{d}{dt}\int_{\mathbb{T}^{3}}u(x,t)\,dx=0$.
 In addition, because of
 Galilean invariance of system \eqref{GNS},  without loss of generality,  we consider initial data of zero average, namely, $\int_{\mathbb{T}^{3}}u(x,0)\,dx=0$. This yields that
  $\int_{\mathbb{T}^{3}}u(x,t)\,dx=0$
  for any $t>0$. A function with vanishing spatial average on torus
guarantees us to  use its Sobolev (Gagliardo-Nirenberg)  inequality  like the whole space case.
 The reader is referred to   \cite{[BO]} for more details.

This  paper is organized as follows. In the  second section, we shall present the definition of weak solutions to \eqref{GNS}  for the initial data with finite-kinetic energy    and
 state our main results.
 Section 3 is devoted to establishing the well-posedness theory with initial data in $H^{\alpha}$ for $ \alpha\geq 5/6$.
 Then, the weak-strong uniqueness in the class $u\in L^{\infty}((0,\,T);\,H^{\alpha})\cap L^{2}((0,\,T);\,H^{2\alpha}) $ is  discussed. Based on this,
 we could obtain the
eventual regularity  and the   Hausdorff dimension of the possible singular time points of weak solutions in  Section 4.
Finally, for completeness, an appendix is dedicated to  proving the existence  of weak solutions to \eqref{GNS} for finite energy initial data.

\vspace{5mm} \noindent\textbf{Notation:}
 The classical   Sobolev norm
$\|\cdot\|_{H^{s}}$
 is defined as $
\norm{f} _{H^{s}}= \sum\limits_{k\in \mathbb{Z}^{3}}
(1+|k|^{2})^{s}|\hat{f}(k)|^{2}  $,   $s\in\mathbb{R}$.
We denote by $\dot{H}^{s}$
homogenous Sobolev spaces  with
the norm $
\norm{f} _{\dot{H}^{s}}= \sum\limits_{k\in \mathbb{Z}^{3}}
|k|^{2s}|\hat{f}(k)|^{2}  .$
We shall denote by $\langle f,\,g\rangle$ the $L^{2}$ inner product of $f$ and $g$. For $p\in [1,\,\infty]$, the notation $L^{p}((0,\,T);X)$ stands for the set of measurable function on the interval $(0,\,T)$ with values in $X$ and $\|f\|_{X}$ belongs to $L^{p}(0,\,T)$. We will use $C$ to denote an absolute
   constant which may be different from line to line unless otherwise stated.

\section{ Statement of the result }
\setcounter{section}{2}\setcounter{equation}{0}
In this section, we begin with  the definition of
 weak solutions to the generalized  Navier-Stokes equations~\eqref{GNS}.
 \begin{definition}\label{1.1}
Let $u(0)\in L^{2}$ be a divergence-free vector.
For any $T>0$, we say that $u$ is a weak solution to
equations \eqref{GNS} with $\alpha>0$,  if \ben[\rm(i)]
\item $u\in L^{\infty}((0,\,T);L^{2}(\Omega))\cap L^{2}((0,\,T);\,H^{\alpha}(\Omega))$; \\
\item $u$~solves (\ref{GNS}) in the sense of distributions
\[
\int_{0}^{T}\langle u(t),\phi_{t}(t)\rangle-\nu\langle\Lambda^{\alpha}u(t),
\Lambda^{\alpha}\phi(t)\rangle
-\langle u(t)\cdot\nabla u(t),\phi(t)\rangle \,dt = -\langle u(0),\phi(0)\rangle,\]
 for all $\phi \in
C_{0}^{\infty}((0,\,T)\times\Omega)$ with $\Div\, \phi=0. $
\item $u$ fulfills the strong energy inequality
\be\label{SEI} \|u(t)\|^{2}_{L^{2}}
+2\nu\int^{t}_{\tau}\|\Lambda^{\alpha} u(s)\|^{2}_{L^{2}
}\,ds\leq \|u(\tau)\|^{2}_{L^{2}},
 \ee
for    $\tau=0$ or almost every $\tau>0 $ and  ~$t\in[\tau,\,T)$.
\een
\end{definition}
The global existence of  weak solutions for the fractional Navier-Stokes  equations \eqref{GNS} could be
proved by the  classical Faedo-Galerkin
argument.   For  reader's convenience, we will sketch its proof
 in Appendix \ref{appendix}.

 The
 strong energy inequality plays an important role in the proof of eventual regularity and  Hausdorff dimension estimate of possible time singular  points. As said before,
it is not clear  whether  the strong energy
inequality of \eqref{GNS} ($\alpha\neq1$)
is valid in the whole space
 since that when we apply  compactness
 theorem to the approximation solution $u^{N}(x,t)$,
 we just obtain a subsequence of $u^{N}(x,t)$ which
 strongly converge to $u$ in $L^{2}_{loc}(\mathbb{R}^{3})$
 for almost every $t\in [0,T]$ as $N$  tends to infinity.
  We refer the reader to Appendix \ref{appendix} for the details. We would like to point out that the usual localization argument for the  proof of the strong energy inequality to the 3d Navier-Stokes equations in  the whole space or the exterior  domain  breaks down
   in the case that  $\alpha\neq1$ in \eqref{GNS}
    due to the presence of the nonlocal derivative.

Now we recall the definition of time singular point  and Hausdorff dimension.

\begin{definition}[Irregular Point, \cite{[Leray]}]
A time point $\tau<\infty$ is said to be an irregular point of the
solution to  the  equations \eqref{GNS}   if $u(s)$ is a regular solution on $ (t,\tau)\times \Omega$, for some $t<\tau $, and
it is impossible
 to extend $u(s)$ to a regular solution on $(t,\tau')$ with $\tau'>\tau$.
\end{definition}

\begin{definition}[Hausdorff measure and Hausdorff dimension, \cite{[Ziemer]}]
\label{HD} For each $\gamma>0,~\varepsilon>0$, and
$E\subset\mathbb{R}^{n}$. Set
$$
\mathcal{H}^{\gamma}(E)=\liminf_{\varepsilon\rightarrow0}
\Big\{\sum^{\infty}_{i=1}\alpha(\gamma)2^{-\gamma}\mathrm{diam}\, (B_{i})^{\gamma}: E\subset\bigcup^{\infty}_{i=1}B_{i}, B_{i} \text{ is an open ball
and} ~ \mathrm{diam}\, B_{i}<\varepsilon\Big\}.
$$
$\mathcal{H}^{\gamma}(E)$ is called the Hausdorff measure of $E$, where $\alpha(\gamma)$ denotes the volume of the unit ball in $\mathbb{R}^{n}$.
For every set $E$, there is a non-negative number, $ d=d(E)$, such
that
$$
\ba
&\mathcal{H}^{\gamma}(E)=0 ~~&\text{if}~~ \gamma>d,\\
&\mathcal{H}^{\gamma}(E)=\infty~~ &\text{if}~~ \gamma<d. \ea$$ the
number $d(E)$ is called the Hausdorff dimension of $E$.
\end{definition}

Let $\mathcal{IR} $ stands for  the possible time irregular
 points of weak solutions to \eqref{GNS} and the rest is denoted by  $\mathcal{R}$.
    Our result is  concerned with the estimate for the Hausdorff dimension of the set $\mathcal{IR} $.
\begin{theorem}\label{Hausdorff}
For $5/6\leq\alpha<5/4$, the $(5-4\alpha)/2\alpha$ dimensional
Hausdorff measure
 of the possible time singular points of weak solutions of \eqref{GNS}
  on the interval $(0,\infty)$ is zero, namely,
$$  \mathcal{H}^{\f{5-4\alpha}{2\alpha}}(\mathcal{IR})=0.$$
\end{theorem}
\begin{remark}In contrast to
 Katz and Pavlovi\'c's result that  the Hausdorrf dimension at the  first time breakdown for the smooth solution to \eqref{GNS}
  is at most $5-4\alpha$ for $1<\alpha< 5/4$ in \cite{[KP]}, our result is concerned with the Hausdorff dimension estimate of the potential time singular set for the finite energy   weak solutions to \eqref{GNS} and  is meaningful  in the case $5/6\leq\alpha<1$. It should be pointed out that the  Hausdorff dimension of the potential time-space singular points for the (suitable) weak solutions to \eqref{GNS} has not been solved except $\alpha=1$.
\end{remark}

Before we show the above theorem, actually, we first prove that
\begin{theorem}\label{eventual regularization}
For $5/6\leq\alpha<5/4$, there exists a constant $T^{\ast}>0$ such that
the weak solutions
 to \eqref{GNS} become   strong solution on $(T^{\ast},\infty)$.
\end{theorem}
\begin{remark}
It is worth   pointing out
 that our results are still valid under the assumption that the weak solutions satisfy the strong energy inequality in the whole space.
\end{remark}
\begin{remark}
According to \eqref{dimensional analysis2}, we notice that the 2d fractional
 Navier-Stokes equations  \eqref{GNS} is also a supercritical system in the case  that $\alpha<1$.
 In the spirit of this paper, one could prove that the upper bound on the Hausdorff  dimension   of the potential time singular points of weak solutions   to \eqref{GNS} for $2/3\leq\alpha<1$ is  $(2-2\alpha)/\alpha$ and
  the eventual regularization of weak solutions is true when the  spatial dimension is two.
  We leave this to the interested reader.
 More generally, the upper bound on the Hausdorff  dimension   of the possible time singular points of weak solutions to the $n$ dimensional generalized  Navier-Stokes equations is $(n+2-4\alpha)/2\alpha$ in the case that  $(n+2)/6\leq\alpha<(n+2)/4$.
\end{remark}

To prove our results, we will employ some fundamental
 strategy in \cite{[Galdi],[H],[Leray],[Scheffer]}.
In Leray's pioneering work \cite{[Leray]}, he proved   the global existence and uniqueness of solution to the 3d Navier-Stokes equations
  with initial velocity $u(0)\in H^{1}$
provided $\|u(0)\|^{1/2}_{L^{2}}\|\nabla u(0)\|^{1/2}_{L^{2}}$ being sufficiently small and  the  local solvability
when the initial
data belongs   to space $H^1$
without   smallness assumption. See \cite{[Galdi],[H]} for a more modern exposition. Leray's eventual regularization of weak solutions relied on the  $H^1$ well-posedness result.
In addition, it is worth mentioning that
\be\label{leray}
\|u(0)\|_{\dot{H}^{1/2}}\leq C \|u(0)\|^{\frac{1}{2}}_{L^{2}}\|\nabla u(0)\|^{\frac{1}{2}}_{L^{2}}
\ee
is invariant under the scaling \eqref{dimensional analysis}.

A key observation  is  an analogous inequality of \eqref{leray}
$$
\|u\|_{\dot{H}^{\frac{5-4\alpha}{2}}}\leq C \| u\|^{\frac{6\alpha-5}{2\alpha}}_{L^{2}}\|\Lambda^{\alpha}u\|
^{\frac{5-4\alpha}{2\alpha}}_{L^{2}},~~5/6\leq\alpha\leq5/4.
$$
 Notice that $\dot{H}^{\frac{5-4\alpha}{2}}$ appearing in the last inequality is the
critical space to (\ref{GNS}) due to the fact that $\|u_{\lambda}\|_{\dot{H}^{\frac{5-4\alpha}{2}}(\mathbb{R}^{3})}
=\|u\|_{\dot{H}^{\frac{5-4\alpha}{2}}(\mathbb{R}^{3})}$ where $u_{\lambda}=\lambda^{2\alpha-1}u(\lambda
x)$, from which we see that the energy space $L^{2}$ is the critical space to (\ref{GNS}) when $\alpha=5/4$. Another interesting  ingredient of the latter inequality is the borderline case that $\alpha=5/6$, which suggests us that $5/6$ is  a  probable endpoint of our discussion. Indeed, we would like to point out that
 $H^{\alpha}$  is the subcritical  space  to (\ref{GNS}) when $\alpha>5/6$ in the sense that $\|u_{\lambda}\|_{\dot{H}^{\alpha}(\mathbb{R}^{3})}$ converges to $0$ as $\lambda$ tends to $0$. Thence,
we obtain a strong solution when we solves \eqref{GNS} with $\alpha\geq 5/6$ in $H^{\alpha}$.

More precisely,
we shall establish  the local existence
 for arbitrarily large data in space $H^{\alpha}$ and
the global
well-posedness of (\ref{GNS}) under the assumption that the 	
quantity
$
\|u(0)\|_{L^{2}}
^{(6\alpha-5)/2\alpha}\|\Lambda^{\alpha}u(0)\|
_{2}^{(5-4\alpha)/2\alpha}$
 is small enough. To this end, we exploit some new estimates to \eqref{GNS} compared to the previous estimates presented in \cite{[Lions2],[Wu2]}.
Next, we show that this strong solution  coincides with the weak solutions associated with the same initial data. Roughly speaking, small solution in $H^{\alpha}$,       weak-strong uniqueness,
  and the strong energy inequality  \eqref{SEI}
 mean the eventual regularization of weak solutions. Thus, we complete the proof of theorem $\ref{eventual regularization}$.

As a  critical by-product of well-posedness result in $H^{\alpha}$,
we could deduce that
the blow-up rate
$$
\|\Lambda^{\alpha}u(t)\|_{L^{2}}
\geq\f{C\nu^{\f{5-2\alpha}{4\alpha}}}{(t_{0}-t)^{\f{6\alpha-5}
{4\alpha}}},~ t<t_{0},~5/6<\alpha<5/4,
$$
 for $t_{0}$    to be a possible irregular  including interval $[0,T^{\ast})$.
This  enables us to achieve the   Hausdorff dimensional estimate
 of  possible time singular  points of  weak solutions on the interval  $(0,\infty)$.

\section{Well-posedness theory  in $H^{\alpha}$  for $\alpha\geq5/6$ and weak-strong uniqueness}

\subsection{Strong solution}
The following is about
 the  existence and uniqueness of the   solution to
 \eqref{GNS} for $\alpha\geq5/6$  with the
initial data in $ H^{\alpha}$.
Notice that
global well-posedness of  \eqref{GNS}  for small initial
 data belonging to $\dot{H}^{5/6}$ has been established in the sense that $u\in L^{\infty}(0,\infty;\dot{H}^{5/6}) \cap
L^{2}(0,\infty;\dot{H}^{5/3})$   by Wu in
 \cite[Theorem 6.1]{[Wu2]}, where we have used the fact that  homogeneous space $\dot{H}^{5/6}$ and  homogeneous Besov space $\dot{B}^{5/6}_{2,2}$ coincide.
 For
the local well-posedness  for large initial data in  $\dot{H}^{5/6}$, see also \cite[Theorem 6.2]{[Wu2]}. The solution constructed by Wu
is also a weak solution  when the initial data belongs to $H^{5/6} $. This fact is clear since the domain considered here is periodic.
 Consequently, we mainly pay our attention on the
 case when $\alpha>5/6$ blow.
\begin{proposition}[Strong solution]\label{local}
For every $u(0)\in H^{\alpha} $(
$5/6<\alpha\leq5/4 $) with divergence-free, there exists a
constant
$$T<\f{\nu^{\f{5-2\alpha}{6\alpha-5}}}
{\f{2\alpha C}{6\alpha-5}\|\Lambda^{\alpha}u(0)\|_{L^{2}}^{\f{4\alpha}{6\alpha-5}}
}$$
such that the system \eqref{GNS} admits a unique solution $u$ in
$$
 L^{\infty}((0,\,T);\,H^{\alpha}) \cap
L^{2}((0,\,T);\,H^{2\alpha}).
$$
Moreover, let $5/6<\alpha<5/4$ and assume that the initial data satisfies
\be\label{con}
 \| u(0)\|^{\frac{6\alpha-5}{2\alpha}}_{L^{2}}\|\Lambda^{\alpha}u(0)\|
^{\frac{5-4\alpha}{2\alpha}}_{L^{2}}< C_{1}^{-1}\nu,
\ee
for  some constat $C_{1}$ depending on the domain $\Omega$. Then $T$ can be chosen as arbitrary positive constants.
\end{proposition}
\begin{remark}\label{rem}
After we finished this paper, we notice that the well-posedness theory  in $\dot{H}^{\frac{5-4\alpha}{2}}$ to equations \eqref{GNS} with $\alpha\geq5/6$ had been established in \cite{[Wu2]}. Compared with the method adopted in \cite{[Wu2]}, our method is elementary and avoids the highly-sophisticated tools from harmonic analysis. More importantly, as a by-product of our analysis,  we could deduce  a necessary condition for some time $t$ to be a potential blow up time  which plays an important role in the proof of Theorem \ref{Hausdorff}. For the detail, see Proposition \ref{coro}.
\end{remark}
\begin{proof}
The existence proof is based on the
 approximate system blow
$$\partial_{t} u^{N}+\nu \Lambda^{2\alpha}u^{N}+P_{N}(u^{N}\cdot\nabla u^{N})+P_{N}\nabla p^{N}=0, ~\text{div}\, u^{N}=0$$
together with the initial condition
$$u^{N}(x,0)=P_{N}u(x,0),$$
where $
 P_{N}f(x)=\mathcal{F}^{-1}(1_{k\leq N}(k)\hat{u}(k))$.
For its detail, see the Appendix \ref{appendix}.

\noindent
  Step 1: Local well-posedness

 Taking the inner product of
 approximate equations with  $u^{N}$ and integrating by parts, we know that
$$
\frac{1}{2}\frac{d}{dt}\|u^{N}(t)\|_{L^{2}}^{2}+\nu\|\Lambda^{\alpha }u^{N}(t)\|_{L^{2}}^{2}=0,
$$
which implies that
\be\label{str1}
\frac{1}{2} \|u^{N}(t)\|_{L^{2}}^{2}+\nu\int^{t}_{0}\|\Lambda^{\alpha }u^{N}(s)\|_{L^{2}}^{2}\,ds=\frac{1}{2} \|u^{N}(0)\|_{L^{2}}^{2}\leq\frac{1}{2} \|u(0)\|_{L^{2}}^{2}.
\ee
 Similarly, multiplying
\eqref{GNS} by $\Lambda^{2\alpha}u$, we arrive at
 \be\label{H1}
\f{1}{2}\f{d}{dt}\|\Lambda^{\alpha}u^{N}(t)\|^{2}_{L^{2}}
+\nu\|\Lambda^{2\alpha}u^{N}(t)\|^{2}_{L^{2}}= \langle
u^{N}(t)\cdot\nabla u^{N}(t),\Lambda^{2\alpha}u^{N}(t)\rangle. \ee
 By virtue of the
Gagliardo-Nirenberg inequality,
$$
\|u^{N}(t)\|_{L^{\f{6}{3-2\alpha}} }\leq
C\|\Lambda^{\alpha}u^{N}(t)\|_{L^{2}},
$$
and
$$
\|\nabla u^{N}(t)\|_{L^{\f{3}{\alpha}} }\leq
C\|\Lambda^{\alpha}u^{N}(t)\|_{L^{2}}^{\f{6\alpha-5}{2\alpha}}
\|\Lambda^{2\alpha}u^{N}(t)\|^{\f{5-4\alpha}{2\alpha}}_{L^{2}}.
$$
we obtain
$$\ba
|\langle u^{N}(t) \cdot\nabla u^{N}(t),\Lambda^{2\alpha}u^{N}(t)\rangle|
&\leq\|u^{N}(t)\|_{L^{\f{6}{3-2\alpha}}}\|\nabla u^{N}(t)\|
_{L^{\f{3}{\alpha}}(\Omega)}
\|\Lambda^{2\alpha}u^{N}(t)\|_{L^{2}}\\
&\leq
C\|\Lambda^{\alpha}u^{N}(t)\|_{L^{2}}
\|\Lambda^{\alpha}u^{N}(t)\|_{L^{2}}^{\f{6\alpha-5}{2\alpha}}
\|\Lambda^{2\alpha}u^{N}(t)\|^{\f{5-2\alpha}{2\alpha}}_{L^{2}}
\\
&\leq
\f{1}{2}\nu\|\Lambda^{2\alpha}u^{N}(t)\|_{L^{2}}^{2}
+C\nu^{\f{2\alpha-5}{6\alpha-5}}
\|\Lambda^{\alpha}u^{N}(t)\|_{L^{2}}^{\f{2(8\alpha-5)}{6\alpha-5}}
,\ea$$
 where we  used   H\"older's inequality and  Young's inequality.

Plugging this   into bound
\eqref{H1} gives
\be\label{grawall}
\f{d}{dt}\|\Lambda^{\alpha}u^{N}(t)\|^{2}_{L^{2}}
+\nu\|\Lambda^{2\alpha}u^{N}(t)\|_{L^{2}}^{2}
\leq C\nu^{\f{2\alpha-5}{6\alpha-5}}
\|\Lambda^{\alpha}u^{N}(t)\|_{L^{2}}^{\f{2(8\alpha-5)}{6\alpha-5}},
\ee
which implies that
$$\ba
\|\Lambda^{\alpha}u^{N}(t)\|_{L^{2}}&\leq
\frac{\|\Lambda^{\alpha}u^{N}(0)\|_{L^{2}}}{\Big(1-\f{2\alpha
C}{6\alpha-5}
\nu^{\f{2\alpha-5}{6\alpha-5}}
\|\Lambda^{\alpha}u^{N}(0)\|^{\f{4\alpha}{6\alpha-5}}_{L^{2}}t\Big)
^{\f{6\alpha-5}{4\alpha}}}\\
&\leq
\frac{\|\Lambda^{\alpha}u(0)\|_{L^{2}}}{\Big(1-\f{2\alpha
C}{6\alpha-5}
\nu^{\f{2\alpha-5}{6\alpha-5}}
\|\Lambda^{\alpha}u(0)\|^{\f{4\alpha}{6\alpha-5}}_{L^{2}}t\Big)
^{\f{6\alpha-5}{4\alpha}}}\\
&\leq
\|\Lambda^{\alpha}u(0)\|_{L^{2}}+
\frac{C\|\Lambda^{\alpha}u(0)\|_{L^{2}}^{2}
\nu^{\frac{2\alpha-5}{4\alpha}}
t^{\frac{6\alpha-5}{4\alpha}}
}{\Big(1-\f{2\alpha
C}{6\alpha-5}
\nu^{\f{2\alpha-5}{6\alpha-5}}
\|\Lambda^{\alpha}u(0)\|^{\f{4\alpha}{6\alpha-5}}_{L^{2}}t\Big)
^{\f{6\alpha-5}{4\alpha}}}
,
\ea$$
where we have used the fact that $\|u^{N}(0)\|_{L^{2}}\leq \|u(0)\|_{L^{2}}$ and the elementary inequality $(a-b)^{\beta}\geq a^{\beta}-b^{\beta} $ for $a>b>0$ and $0<\beta<1$.

Thus, there exists a constat $T:=T_*$ with
$$T_{\ast}=\f{\nu^{\f{5-2\alpha}{6\alpha-5}}}
{\f{2\alpha C}{6\alpha-5}
\|\Lambda^{\alpha}u(0)\|_{L^{2}}^{\f{4\alpha}{6\alpha-5}}
}.$$
such that
$$
\|\Lambda^{\alpha}u^{N}(t)\|^{2}_{L^{2}}\leq C,\ \
t\in[0,T_{\ast}),$$
 where the constant $C$ is independent of $N$.
Therefore, we deduce the uniformly bounded estimate that \be\label{estimate1}
\|\Lambda^{\alpha}u^{N}(t)\|^{2}_{L^{2}}
+\nu\int^{t}_{0}\|\Lambda^{2\alpha}u^{N}(s)\|^{2}_{L^{2}}\,ds\leq
C,~t< T^{\ast}. \ee
Furthermore, by \eqref{str1},  for any $g\in L^{2}([0,\,t];\,L^{2})$, note that
$$\ba
\int^{t}_{0}\langle\Lambda^{2\alpha}u^{N}(s), \,g (s)\rangle \,ds&\leq \int^{t}_{0}\|\Lambda^{\alpha}u^{N}(s)\|_{L^{2}}\|\Lambda^{\alpha}g(s) \|_{L^{2}}\,ds\\
&\leq \Big(\int^{t}_{0}\|\Lambda^{\alpha}u^{N}(s)\|_{L^{2}}^{2}\,ds\Big)^{1/2}
\Big(\int^{t}_{0}\|\Lambda^{\alpha}g(s)\|_{L^{2}}^{2}\,ds\Big)^{1/2}\\
&\leq  C\|u(0)\|_{L^{2}} \| g(s)\|_{L^{2}([0,\,t];\,L^{2})},
\ea$$
which yields that $\Lambda^{2\alpha}u^{N}\in  L^{2}([0,\,t];\,L^{2})$.

By the   Gagliardo-Nirenberg inequality, we have
$$
\|\nabla u^{N}\|_{L^{\f{3}{\alpha}}}\leq
C\|\Lambda^{\alpha}  u^{N}\|_{L^{2}}^{\f{6\alpha-5}{2\alpha}}
\|\Lambda^{2\alpha}u^{N}\|^{\f{5-4\alpha}{2\alpha}}_{L^{2}},
$$
and
$$
\|u^{N} \|_{L^{\frac{6}{3-2\alpha}}}\leq C\|\Lambda^{\alpha}u^{N} \|_{L^{2}},
$$
and \eqref{str1}, we infer that
\be\label{ab1}\ba
\int^{t}_{0}\langle P_{N}(u^{N}\cdot\nabla u^{N}),g \rangle \,ds&\leq
\int^{t}_{0}\|u^{N}\|_{L^{\frac{6}{3-2\alpha}}}\|\nabla u^{N}\|_{L^{\frac{3}{\alpha}}}
\|g\|_{L^{2}}\,ds\\
 &\leq \int^{t}_{0}\|\Lambda^{\alpha}u^{N}\|_{L^{2}}
 \|\Lambda^{\alpha}u^{N}\|_{L^{2}}^{\frac{6\alpha-5}{2\alpha}}
\|\Lambda^{2\alpha}u^{N}\|_{L^{2}}^{\frac{5-4\alpha}{2\alpha}}
\|g\|_{L^{2}}\,ds \\
 &\leq    C \Big(\int^{t}_{0}\|\Lambda^{\alpha}u^{N}\|^{2}_{L^{2}}
 \|\Lambda^{\alpha}u^{N}\|_{L^{2}}^{\frac{6\alpha-5}{\alpha}}
\|\Lambda^{2\alpha}u^{N}\|_{L^{2}}^{\frac{5-4\alpha}{\alpha}}
 ds \Big)^{\frac{1}{2}}\Big(\int^{t}_{0}\| g\|^{2}_{L^{2}}\, ds \Big)^{\frac{1}{2}}\\
 &\leq    C \|u(0)\|_{L^{2}}\Big(\int^{t}_{0}\|\Lambda^{2\alpha}u^{N}\|_{L^{2}}^{\frac{5-4\alpha}{\alpha}}
 \,ds \Big)^{\frac{1}{2}}
 \| g\|_{L^{2}([0,\,t];\,L^{2})}
 \\
 &\leq  C(T_{\ast})\|u(0)\|_{L^{2}}^{2}\| g\|_{L^{2}([0,\,t];\,L^{2})},
\ea\ee
where we have used \eqref{str1} and \eqref{estimate1}.
This in turn implies that
$P_{N}(u^{N}\cdot\nabla u^{N})\in L^{2}([0,t]; L^{2})$ since $5/6\leq\alpha\leq 5/4$.

Using the pressure equation
$$
\Delta p^{N}=-\sum_{i,j}\partial_{i}\partial_{j}(u_{i}^{N}u_{j}^{N}).
$$
 and
 the boundedness of Riesz transforms on $L^q$
for any $1 < q < \infty$, we  have
$\|P_{N}\nabla p^{N}\|_{ L^{2}}\leq \|u^{N}\cdot\nabla u^{N}\|_{ L^{2}}$, from which,  exactly as in the derivation of \eqref{ab1}, yields that
$P_{N}\nabla p^{N}\in L^{2}([0,\,t]; \,L^{2})$.

Collecting the uniform estimate and
recalling that $\partial_{t}u^{N}=-\nu\Lambda^{2\alpha}u^{N}-P_{N}(u\cdot\nabla u^{N})-P_{N}\nabla p_{N}  $,
we obtain the desired estimate $\partial_{t}u^{N}\in L^{2}([0,\,t];\,L^{2})$.  This together with \eqref{estimate1},
by means of  Aubin-Lions Lemma (\cite[Theorem 2.1, p.184]{[Teman2]}),  we could claim that
$u_{N}$ strongly converge to $u$ as $ N$ tends to infinity in $L^{2}((0,\,T);\, H^{\alpha-\epsilon})$ for any $\epsilon>0$. By Fatou's Lemma, we get that $u\in L^{\infty}((0,\,T);\,H^{\alpha}) \cap
L^{2}((0,\,T);\,H^{2\alpha})$. We finish the proof of local existence. The
  uniqueness is a consequence of  the weak-strong uniqueness Proposition \ref{the1.1} whose   proof  is postponed to the next subsection \ref{ws}.

\noindent
Step 2: Global well-posedness for small initial data.

It suffices to show the uniform boundedness of  $\|\Lambda^{\alpha}u^{N}(t)\|^{2}_{L^{2}} $  where $t\in[0,T']$ for any given constant $T'$. Indeed, having this uniform estimate in hands, by means of \eqref{grawall}, one immediately obtains the uniform estimate of $\int^{t}_{0}\|\Lambda^{2\alpha}u^{N}(s)\|_{L^{2}}^{2}\,ds$ for $t<T'$.
The rest part of passing to the limit of the approximations solution is analogous to the one of local well-posedness.

To this end, using once again
the Gagliardo-Nirenberg inequality,   we obtain
$$
\|u^{N}\|_{L^{\f{3}{2\alpha-1}}}\leq
C\|u^{N}\|^{\f{6\alpha-5}{2\alpha}}_{L^{2}}
\|\Lambda^{\alpha}u^{N}\|^{\f{5-4\alpha}{2\alpha}}_{L^{2}}
,~\text{and}~
\|\nabla u^{N}\|_{L^{\f{6}{5-4\alpha}}}\leq
C\|\Lambda^{2\alpha}u^{N}\|_{L^{2}},$$ which yields
$$\ba
|\langle u^{N} \cdot\nabla u^{N},\,\Lambda^{2\alpha}u^{N}\rangle| &\leq
\|u^{N}\|_{L^{\f{3}{2\alpha-1}}} \|\nabla
u^{N}\|_{L^{\f{6}{5-4\alpha}}}
\|\Lambda^{2\alpha}u^{N}\|_{L^{2}}\\
&\leq C_{1}\|u^{N}\|^{\f{6\alpha-5}{2\alpha}}_{L^{2}}
\|\Lambda^{\alpha}u^{N}\|^{\f{5-4\alpha}{2\alpha}}_{L^{2}}
\|\Lambda^{2\alpha}u^{N}\|^{2}_{L^{2}}. \ea$$
Substituting the latter inequality  into \eqref{H1}, we get  \be\label{estimate2}
\f{1}{2}\f{d}{dt}\|\Lambda^{\alpha}u^{N}(t)\|^{2}_{L^{2}}
+\Big(\nu-C_{1}\|u^{N}(t)\|^{\f{6\alpha-5}{2\alpha}}_{L^{2}}
\|\Lambda^{\alpha}u^{N}(t)\|
_{_{L^{2}}}^{\f{5-4\alpha}{2\alpha}}\Big)
\|\Lambda^{2\alpha}u^{N}(t)\|^{2}_{L^{2}}\leq0. \ee
Assume for a while we have proved that
\be\label{assu}
 \| u^{N}(t)\|^{\frac{6\alpha-5}{2\alpha}}_{L^{2}}
 \|\Lambda^{\alpha}u^{N}(t)\|
^{\frac{5-4\alpha}{2\alpha}}_{L^{2}}< C_{1}^{-1}\nu, ~~\text{for each}~t\in[0,T').
\ee
Therefore, it follows form  \eqref{estimate2} that
$\f{d}{dt}\|\Lambda^{\alpha}u^{N}(t)\|^{2}_{L^{2}}\leq0$ ($t\in[0,T')$), which yields the desired estimate.

Now we need to prove the equality \eqref{assu} we have assumed.
By the hypothesis on initial condition \eqref{con}, we could suppose that $T_{1}<T'$ is the first time such that
\be\label{assume}
\| u^{N}(T_{1})\|^{\frac{6\alpha-5}{2\alpha}}_{L^{2}}
\|\Lambda^{\alpha}u^{N}(T_{1})\|
^{\frac{5-4\alpha}{2\alpha}}_{L^{2}}= C_{1}^{-1}\nu.
\ee
Consequently, for any
$  s\in [0,T_{1}],$ we have $\| u^{N}(s)\|^{\frac{6\alpha-5}{2\alpha}}_{L^{2}}\|\Lambda^{\alpha}u^{N}(s)\|
^{\frac{5-4\alpha}{2\alpha}}_{L^{2}}\leq C_{1}^{-1}\nu.$\\
Thanks to \eqref{str1} and \eqref{estimate2}, we obtain
$$\|u^{N}(T_{1})\|_{L^{2}}\leq\|u^{N}(0)\|_{L^{2}}\leq \|u (0)\|_{L^{2}},$$
and
$$\|\Lambda^{\alpha}u^{N}(T_{1})\|_{L^{2}}\leq \|\Lambda^{\alpha}u^{N}(0)\|_{L^{2}}\leq\|\Lambda^{\alpha}u(0)\|_{L^{2}},$$
 since
$\frac{d}{ds}\|\Lambda^{\alpha}u(T_{1})\|_{L^{2}}^{2}\leq0$.

However, through a simple calculation, we see that
$$
\nu- C_{1}\| u^{N}(T_{1})\|^{\frac{6\alpha-5}{2\alpha}}_{L^{2}}
\|\Lambda^{\alpha}u^{N}(T_{1})\|
^{\frac{5-4\alpha}{2\alpha}}_{L^{2}}\geq \nu-C_{1} \| u^{N}(0)\|^{\frac{6\alpha-5}{2\alpha}}_{L^{2}}
\|\Lambda^{\alpha}u^{N}(0)\|
^{\frac{5-4\alpha}{2\alpha}}_{L^{2}}>0.
$$
which contradicts \eqref{assume}.
Thus, the claim \eqref{assu} is proved.
\end{proof}

\begin{remark}\label{meaning}
A straightforward consequence of
 the above proof is that we provide an alternative approach to show Lions's global  solvability to the case $\alpha=5/4$.
Indeed,  it follows from the equation \eqref{grawall} and the
Gronwall's inequality that
$$
\|\Lambda^{5/4}u^{N}(t)\|_{L^{2}}^{2}\leq \|\Lambda^{5/4}u^{N}(0)\|_{L^{2}}^{2}
e^{C\int_{0}^{t}\|\Lambda^{5/4}u^{N}(s)\|_{L^{2}}^{2}\,ds}
\leq \|\Lambda^{5/4}u(0)\|_{L^{2}}^{2}
e^{C\|u(0)\|_{L^{2}}^{2}}.
$$	
Thus, there exists a
global strong solution to \eqref{GNS}   without without   any smallness restriction on the initial data belonging to space $H^{5/4}$ when $\alpha=5/4$. (It is well known that  the   finite weak solution $u\in L ^{\infty}((0,\,T);\, L^2)\cap L ^{2}((0,\,T);\, H^{5/4})$ is a regular solution to the \eqref{GNS} for $\alpha=5/4$. Here, we show that the weak solution has higher regularity when the initial data is more regular in the case $\alpha=5/4$.)
\end{remark}

Next we shall prove that  the strong solution $u\in L^{\infty}((0,\,T);\,H^{\alpha})\cap L^{2}((0,\,T);\,H^{2\alpha}) $ constructed above and the case $\alpha=5/6$ in \cite{[Wu2]}
coincides with
  the weak solution $v\in L^{\infty}((0,\,T);L^{2})\cap L^{2}((0,\,T);\,H^{\alpha}) $ associated  with the same initial data for $5/6\leq\alpha<5/4$. Here, we basically follow the pathway of \cite{[CMZ],[Dubois],[Galdi],[M],[Serrin2]} to obtain the weak-strong uniqueness in the class $u\in L^{\infty}((0,\,T);\,H^{\alpha})\cap L^{2}((0,\,T);\,H^{2\alpha}) $.

\subsection{Weak-strong uniqueness}\label{ws}
\begin{proposition}\label{the1.1}
Let $5/6\leq\alpha\leq5/4$.  From  Proposition \ref{local} in the last subsection and Theorem 6.2 in \cite{[Wu2]}, there exist  a  solution  $u\in L^{\infty}((0,\,T);\,H^{\alpha})\cap L^{2}((0,\,T);\,H^{2\alpha})$  to \eqref{GNS} for
a constant $T>0$ associated with  data $u(0)$ in $H^{\alpha}$. Let $v$ be any weak solution to \eqref{GNS} with initial condition $u(0)$.
 Then $u=v$ for a.e. $x\in \Omega, t>0$.
\end{proposition}
In order to show Proposition \ref{the1.1}, we need
 the following  lemmas.
\begin{lemma}\label{lemL2}
Let $v$ be a weak solution to \eqref{GNS} in $(0,\,T)\times
\Omega$. Then $v$ can be redefined  on a set of zero
Lebesgue measure in such a way that $v(t)\in L^{2} $ for all
$t\in(0,\,T)$ and satisfies the identity \be\label{I}
\int^{t}_{s}\langle
v,\varphi_{\tau}\rangle-\nu\langle\Lambda^{\alpha}v,\,\Lambda^{\alpha}\varphi\rangle
-\langle v\cdot\nabla v,\varphi\rangle  d\tau=\langle
v(t),\varphi(t)\rangle  -\langle v(s),\varphi(s)\rangle,\ee for all
$s\in [0,t],\, t<T$ and all $\varphi\in
C^{\infty}((0,\,T)\times C^{\infty}_{0})$.
\end{lemma}
The proof of  this lemma is similar to the classical
Navier-Stokes equations (\cite{[Galdi],[M],[Serrin2]}) and we omit
the details here. The following  is about an approximation lemma.
\begin{lemma}[\cite{[M]}, Lemma 2.1]\label{lem2.1}
Suppose that $X$ is a Banach space, $w\in L^{q}((0,\,T);X),1\leq
q<\infty, w_{\rho}(s)= \int^{t}_{0}J_{\rho}(s-\tau)w(\tau)d\tau$,
where $J_{\rho}(t)=\rho^{-1}J(t/\rho)$ and
 $J$ is an even nonnegative smooth function with
 $\int^{\infty}_{-\infty}J(s)\,ds=1$ and $J\in C^{\infty}_{0}(-1,\,1)$. Then $w_{\rho}\in C^{1}([0,\,t];\,X)$ and
\[
\lim_{\rho\rightarrow 0}\|w_{\rho}-w\|_{L^{q}((0,\,T);\,X)}=0.
\]
\end{lemma}
The main proof of Proposition \ref{the1.1}  relies on the  following key lemma.
\begin{lemma}\label{prop2.5}
Suppose that $u$ and $v$  are  two weak solutions of (\ref{GNS}) in
Proposition \ref{the1.1} and let $w=u-v$.
Then for any $T>0$,  there holds \be\label{2.1}
 \langle  u(t), v(t)\rangle  +2\nu\int^{t}_{0}\langle  \Lambda^{\alpha}u(s),\,\Lambda^{\alpha}v(s)\rangle \, ds
=\langle u(0),\,v(0)\rangle-\int^{t}_{0}\langle  w(s)\cdot\nabla
u(s),w(s)\rangle \, ds.\ee
\end{lemma}
\begin{proof}
It follows from Lemma \ref{lem2.1}
that $u_{\rho} \in
H^{1}([0,T];\,H^{2\alpha})$,
$v_{\rho} \in
H^{1}([0,T];\,H^{\alpha})$ ~and
$$\lim_{\rho\rightarrow0}\|u_{\rho}-u\|_{L^{2}((0,\,T);\,H^{2\alpha})}=0,~
\lim_{\rho\rightarrow0}\|v_{\rho}-v\|_{L^{2}((0,\,T);\,H^{\alpha})}=0.
$$
 Due to the fact that $C^{\infty}_{0}(\Omega)$ is dense
 in $H^{\alpha}(\Omega)$ ($H^{2\alpha}(\Omega)$) and Lemma  2.2 in \cite{[M]}, there exist $u^{k}_{\rho},\,v^{k}_{\rho}\in C^{\infty}([0,\,t];\,C^{\infty}_{0})$ such that
\be\label{app} \lim_{k\rightarrow \infty}
\|u^{k}_{\rho}-u_{\rho}\|_{H^{1}([0,\,t];\,H^{2\alpha})}=0,~
\lim _{k\rightarrow \infty}
\|v^{k}_{\rho}-v_{\rho}\|_{H^{1}([0,\,t];\,H^{\alpha})}=0.
\ee Choosing $u^{k}_{\rho}$ and $v^{k}_{\rho}$ as test functions in
\eqref{I} respectively, one has
\be\ba\label{14}\int^{t}_{0}\langle v,\partial_{s}u_{\rho}^{k}\rangle-
\nu\langle\Lambda^{\alpha}v,\Lambda^{\alpha}u_{\rho}^{k}\rangle-
\langle v\cdot\nabla v,\,u_{\rho}^{k}\rangle \, ds&=\langle v(t),u_{\rho}^{k}(t)\rangle-\langle v(0),u_{\rho}^{k}(0)\rangle,\\
\int^{t}_{0}\langle
u,\partial_{s}v_{\rho}^{k}\rangle-\nu\langle\Lambda^{\alpha}u,
\Lambda^{\alpha}v_{\rho}^{k}\rangle- \langle u\cdot\nabla
u,\,v_{\rho}^{k}\rangle \, ds&= \langle
u(t),v_{\rho}^{k}(t)\rangle-\langle u(0),v_{\rho}^{k}(0)\rangle.
\ea\ee
Now we take  the limit in the above equations.
 By \eqref{app}, it is direct to get
\[\ba
 \lim _{k\rightarrow \infty}\int^{t}_{0}\langle v,\partial_{s}u_{\rho}^{k}\rangle \, ds
 =\int^{t}_{0}\langle v,\partial_{s}u_{\rho}\rangle \, ds ~\text{and}~ \lim _{k\rightarrow \infty}\int^{t}_{0}\langle u,\partial_{s}v_{\rho}^{k}\rangle \, ds
 =\int^{t}_{0}\langle u,\partial_{s}v_{\rho}\rangle \, ds.
\ea\]
Noting that the function $J$ defined in Lemma \ref{lem2.1} is
an even function, we have
\[\ba
\int^{t}_{0}\langle v,\partial_{s}u_{\rho}\rangle
ds+\int^{t}_{0}\langle
u,\partial_{s}v_{\rho}\rangle \,ds&=\int^{t}_{0}\int^{t}_{0}\partial_{s}
J_{\rho}(s-\tau)\langle u(\tau),\,v(s)\rangle \,dsd\tau+\int^{t}_{0}\langle
u,\partial_{s}v_{\rho} \rangle \,ds\\
&=-\int^{t}_{0}\int^{t}_{0}\partial_{\tau}
J_{\rho}(s-\tau)\langle u(\tau),\,v(s)\rangle \,dsd\tau+\int^{t}_{0}\langle
u,\partial_{s}v_{\rho} \rangle \,ds\\
&=-\int^{t}_{0}\int^{t}_{0}\partial_{\tau}
J_{\rho}(\tau-s)\langle u(\tau),\,v(s)\rangle \,dsd\tau+\int^{t}_{0}\langle
u,\partial_{s}v_{\rho} \rangle \,ds,\\
&=-\int^{t}_{0}\langle \partial_{\tau}v_{\rho},u\rangle d\tau+\int^{t}_{0}\langle
u,\partial_{s}v_{\rho} \rangle \,ds=0.
\ea\]
With the help of Lemma \ref{lem2.1} and \eqref{app}, we obtain
\[
  \lim _{\rho\rightarrow 0}\lim _{k\rightarrow \infty}\int^{t}_{0}\langle  \Lambda^{\alpha}v,\Lambda^{\alpha}u_{\rho}^{k}\rangle \, ds=
\int^{t}_{0}\langle  \Lambda^{\alpha}v,\Lambda^{\alpha}u\rangle \, ds,
\]
\[
\lim _{k\rightarrow \infty}\langle
v(t),u_{\rho}^{k}(t)\rangle \,ds=\langle v(t),u_{\rho}(t)\rangle,
\]
where the notation $\lim\limits _{\rho\rightarrow 0}\lim\limits _{k\rightarrow \infty}$ means that one first passes to the limit of $k$ and then that of $\rho$.

Thanks to  the  $L^{2}$ weak continuity of weak solutions
proved in Theorem \ref{weak solution}  in appendix \ref{appendix}, we deduce
\[\ba
\langle v(t),u_{\rho}(t)\rangle
&=\langle v(t),\int_{0}^{t}\frac{1}{\rho}J(\frac{t-s}{\rho})u(s)ds\rangle\\
&=\langle v(t),\int_{0}^{\frac{t}{\rho}}J(s')u(t-\rho s')ds'\rangle\\
&=\langle v(t),\int_{0}^{1}J(s')u(t)ds\rangle+\int_{0}^{1}J(s')\langle v(t),u(t-\rho s')-u(t)\rangle \,ds'\\
&=\f{1}{2}\langle v(t),u(t)\rangle  +o(\rho),\ea\]
where we have used the fact that $\int^{1}_{0}J_{\rho}(z)dz=\f{1}{2}$ and $\rho<t$, which yields
\[
\lim _{\rho\rightarrow 0}\langle v(t),u_{\rho}(t)\rangle
=\f{1}{2}\langle v(t),u(t)\rangle.
\]
Likewise,
$$\ba
&\lim _{\rho\rightarrow 0}\lim _{k\rightarrow
\infty}\int^{t}_{0}\langle
\Lambda^{\alpha}u,\,\Lambda^{\alpha}v_{\rho}^{k}\rangle \, ds=
\int^{t}_{0}\langle  \Lambda^{\alpha}u,\,\Lambda^{\alpha}v\rangle \, ds,\\
&\lim _{\rho\rightarrow 0}\lim _{k\rightarrow \infty}\langle u(t),\,v_{\rho}^{k}(t)\rangle  =\f{1}{2}\langle v(t),\,u(t)\rangle  ,\\
&\lim _{\rho\rightarrow 0}\lim _{k\rightarrow \infty}\langle u(0),\,v_{\rho}^{k}(0)\rangle=\f{1}{2}\langle u(0),\,v(0)\rangle,\\
&\lim _{\rho\rightarrow 0}\lim _{k\rightarrow \infty}\langle v(0),\,u_{\rho}^{k}(0)\rangle=\f{1}{2}\langle v(0),\,u(0)\rangle.\\
\ea$$
 Consequently, taking the limit in \eqref{14} leads to
\be\label{3.5}
 \langle  u(t), \,v(t)\rangle  +2\int^{t}_{0}\langle  \Lambda^{\alpha}u,\,\Lambda^{\alpha}v\rangle \, ds
=\langle u(0),\,v(0)\rangle-\lim _{\rho\rightarrow 0}\lim
_{k\rightarrow \infty}\int^{t}_{0}\Big[ \langle u\cdot\nabla
u,\,v_{\rho}^{k}\rangle \, ds +\langle v\cdot\nabla
v,\,u_{\rho}^{k}\rangle \Big] \,ds .\ee
It remains to pass to the limit in nonlinear terms in \eqref{3.5}.
In order to do this, using the Gagliardo-Nirenberg
inequality
$$
\|\nabla u\|_{L^{\f{3}{\alpha}}}\leq
C\|\Lambda^{\alpha}u\|_{L^{2}}^{\f{6\alpha-5}{2\alpha}}
\|\Lambda^{2\alpha}u\|^{\f{5-4\alpha}{2\alpha}}_{L^{2}},
$$
and
$$
\|v_{\rho}^{k}-v \|^{2}_{L^{\frac{6}{3-2\alpha}}}\leq C\|\Lambda^{\alpha}(v_{\rho}^{k}-v ) \|^{2}_{L^{2}},
$$
we see that
\be\label{hao1}\ba
&\int^{t}_{0}\langle u\cdot\nabla u,  v_{\rho}^{k}-v \rangle \,ds\\
\leq& \Big(\int^{t}_{0}\|u\cdot \nabla u\|^{2}_{L^{\frac{6}{3+2\alpha}}}\,ds \Big)^{\frac{1}{2}}
\Big(\int^{t}_{0}\|v_{\rho}-v\|^{2}_{L^{\frac{6}{3-2\alpha}}}\,ds \Big)^{\frac{1}{2}}\\
\leq&\Big(\int^{t}_{0}\|u\|^{2}_{L^{2}}\|\nabla u\|^{2}_{\frac{3}{\alpha}}\,ds \Big)^{\frac{1}{2}}
\Big(\int^{t}_{0}\|v_{\rho}^{k}-v \|^{2}_{L^{\frac{6}{3-2\alpha}}}\,ds \Big)^{\frac{1}{2}}\\
\leq&\Big(\int^{t}_{0}
\|u\|^{2}_{L^{2}}\|\Lambda^{\alpha}u(t)\|
_{L^{2}}^{\f{6\alpha-5}{\alpha}}
\|\Lambda^{2\alpha}u(t)\|^{\f{5-4\alpha}{\alpha}}
_{L^{2}}\,ds\Big)^{\frac{1}{2}}
\Big(\int^{t}_{0}\|\Lambda^{\alpha}(v_{\rho}^{k}-v ) \|^{2}_{L^{2}}\,ds \Big)^{\frac{1}{2}}
,\ea\ee
where we have used the H\"older inequality.

Utilizing the   H\"older inequality and the Gagliardo-Nirenberg
inequality
$$
\|v \|_{L^{\frac{12}{4\alpha+1}}}\leq C \|v\|_{L^{2}}^{\frac{8\alpha-5}{4\alpha}}
\|\Lambda^{\alpha}v\|_{L^{2}}^{\frac{5-4\alpha}{4\alpha}},
$$
and
$$
\|\nabla( u_{\rho}^{k}-u)\|_{L^{\frac{6}{5-4\alpha}}}\leq C\|\Lambda^{2\alpha}( u_{\rho}^{k}-u)\|_{L^{2}},
$$
by  integration by parts, we infer  that
\be\label{hao2}\ba
\int^{t}_{0}\langle v\cdot\nabla v, u_{\rho}^{k}-u\rangle \,ds
&=-\int^{t}_{0}\langle v\cdot\nabla( u_{\rho}^{k}-u),
  v\rangle \,ds\\
&\leq \int^{t}_{0}\|v^{2}\|_{L^{\frac{6}{4\alpha+1}}}\|\nabla( u_{\rho}^{k}-u)\|_{L^{\frac{6}{5-4\alpha}}}\,ds\\
&\leq \int^{t}_{0}\|v\|^{2}_{L^{\frac{12}{4\alpha+1}}}\|\nabla( u_{\rho}^{k}-u)\|_{L^{\frac{6}{5-4\alpha}}}\,ds\\
&\leq \Big(\int^{t}_{0}\|v\|_{L^{2}}^{\frac{8\alpha-5}{\alpha}}
\|\Lambda^{\alpha}v\|_{L^{2}}^{\frac{5-4\alpha}{\alpha}}\,ds\Big)^{\frac{1}{2}}
\Big(\int^{t}_{0}\|\Lambda^{2\alpha}( u_{\rho}^{k}-u)\|^{2}_{L^{2}}\,ds\Big)^{\frac{1}{2}}
.\ea\ee
In the light of
$\f{5-4\alpha}{\alpha}\leq2$, we can pass the limit in \eqref{hao1} and \eqref{hao2}, namely,
\begin{align}\lim _{\rho\rightarrow 0}\lim _{k\rightarrow \infty}
\int^{t}_{0}\langle  u\cdot\nabla u,\,v_{\rho}^{k}\rangle
ds&=\int^{t}_{0}\langle  u\cdot\nabla u,\,v\rangle \, ds,\label{2.20}\\
\lim _{\rho\rightarrow 0}\lim _{k\rightarrow \infty}
\int^{t}_{0}\langle  v\cdot\nabla v,\,u_{\rho}^{k}\rangle \, ds&=-\int^{t}_{0}\langle  v\cdot\nabla u,\,v\rangle \, ds.\label{2.19}
\end{align}
Consequently, we have
\be\label{2000} \lim _{\rho\rightarrow 0}\lim
_{k\rightarrow \infty}\int^{t}_{0}\big[ \langle u\cdot\nabla
u,\,v_{\rho}^{k}\rangle  +\langle v\cdot\nabla v,\,u_{\rho}^{k}\rangle
\big] \,ds=\int^{t}_{0}\langle w\cdot\nabla u,\,v\rangle \, ds. \ee
Using the integration by parts, we know that
$$
\int^{t}_{0}\langle  w\cdot\nabla u,\,u_{\rho}^{k}\rangle\,
ds=-\int^{t}_{0}\langle w\cdot\nabla u_{\rho}^{k},u\rangle \, ds.
$$
Exactly as in the derivation of \eqref{2.20} and \eqref{2.19}, we deduce that
$$
\lim _{\rho\rightarrow 0}\lim _{k\rightarrow \infty}
\int^{t}_{0}\langle  w\cdot\nabla u,\,u_{\rho}^{k}\rangle\,  ds=-
\int^{t}_{0}\langle  w\cdot\nabla u,\,u\rangle \, ds,
$$
and
$$
\lim _{\rho\rightarrow 0}\lim _{k\rightarrow \infty}
\int^{t}_{0}\langle  w\cdot\nabla u_{\rho}^{k},u\rangle\,  ds=
\int^{t}_{0}\langle  w\cdot\nabla u,\,u\rangle \, ds,
$$
which implies
\be\label{2.22} \int^{t}_{0}\langle  w\cdot\nabla
u,\,u\rangle  \,ds=0. \ee

Combining \eqref {3.5}  and \eqref{2000} with \eqref{2.22}, we obtain
\eqref{2.1}.
  \end{proof}
\begin{proof}[Proof of Proposition \ref{the1.1}]
In view of \eqref{2.1} and the energy inequality corresponding to $s=0$ in \eqref{SEI2},  straightforward calculations show that
 \be\label{21}\ba
&\|w(t)\|^{2}_{L^{2}}+2\int^{t}_{0}\|\Lambda^{\alpha}w(s)\|^{2}_{L^{2}}
\,ds\\
=&\|u(t)\|^{2}_{L^{2}}+\|v(t)\|^{2}_{L^{2}}
-2\langle u(t),\,v(t)\rangle +2\int^{t}_{0}\|\Lambda^{\alpha}u(s)\|^{2}_{L^{2}}
\,ds\\
&+2\int^{t}_{0}\|\Lambda^{\alpha}v(s)\|^{2}_{L^{2}}\,
ds-4\int^{t}_{0}\langle \Lambda^{\alpha}u(s),\Lambda^{\alpha}v(s)\rangle \,ds\\
\leq& \|u(0)\|^{2}_{L^{2}}+\|v(0)\|^{2}_{L^{2}}
-2\langle u(0),\, v(0)\rangle-2\int^{t}_{0}\langle w(s)\cdot\nabla w(s), \,u(s)\rangle \,ds\\
=&-2\int^{t}_{0}\langle w(s)\cdot\nabla u(s),\,w(s)\rangle\, \,ds.\ea\ee
By the H\"older inequality,  the Gagliardo-Nirenberg inequality used in \eqref{hao2} and the Young's inequality,
we get
$$\ba
-\int^{t}_{0}\langle w(s)\cdot\nabla u(s),
  \,w(s)\rangle\, ds
&\leq \int^{t}_{0}
\|w^{2}(s)\|_{L^{\frac{6}{4\alpha+1}}}\|\nabla u(s)\|_{L^{\frac{6}{5-4\alpha}}}\,ds\\
&\leq \int^{t}_{0}  \|w(s)\|_{L^{2}}^{\frac{8\alpha-5}{2\alpha}}
\|\Lambda^{\alpha}w(s)\|_{L^{2}}^{\frac{5-4\alpha}{2\alpha}}
\|\Lambda^{2\alpha}u(s)\|_{L^{2}}\,ds\\
&\leq \int^{t}_{0} \|\Lambda^{\alpha}w(s)\|_{L^{2}}^{2}\,ds+C\int^{t}_{0}\|w(s)\|^{2}_{L^{2}}
\|\Lambda^{2\alpha}u(s)\|_{L^{2}}^{\frac{4\alpha}{8\alpha-5}}\,ds
.\ea$$
This  together with   \eqref{21}  yields
\be\label{22}
\|w(t)\|^{2}_{L^{2}}+\int^{t}_{0}
\|\Lambda^{\alpha}w(s)\|^{2}_{L^{2}}\,
ds\leq C\int^{t}_{0}\|w(s)\|^{2}_{L^{2}}
\|\Lambda^{2\alpha}u(s)\|_{L^{2}}^{\frac{4\alpha}{8\alpha-5}}\,
ds.
\ee
 Thus, making use of
Gronwall's Lemma, we  accomplish the proof.
\end{proof}

\section{Eventual regularity and  Hausdorff dimension estimate}
\setcounter{section}{4}\setcounter{equation}{0}
We shall basically follow  the pathway of \cite[Section 6]{[Galdi]} to complete the proof of Theorem \ref{Hausdorff}  and
\ref{eventual regularization}  in this section.
\subsection{Eventual regularity}
This subsection focuses on
the eventual regularity of weak solutions, which means that there
 exists a $T^{\ast}>0$ such that every weak solution $u(t)$  is
a strong solution on $(T^{\ast},\infty)$.
\begin{proof}[Proof of Theorem \ref{eventual regularization}]
 By the global well-posedness result for small
solution in Proposition \ref{local} and Theorem 6.1
 in \cite{[Wu2]},
 it suffices to prove that there exists a $T^*>0$ (maybe large) such that
 \be\label{E1}
 \| u(T^{\ast})\|^{\frac{6\alpha-5}{2\alpha}}_{L^{2}}
 \|\Lambda^{\alpha}u(T^{\ast})\|
^{\frac{5-4\alpha}{2\alpha}}_{L^{2}}< C_{1}^{-1}\nu.
 \ee
 The proof of \eqref{E1} is easy.
Otherwise,  for any $ t\in(0,\infty)$, we have
$$
 \| u(t)\|^{\frac{6\alpha-5}{2\alpha}}_{L^{2}}
 \|\Lambda^{\alpha}u(t)\|
^{\frac{5-4\alpha}{2\alpha}}_{L^{2}}\geq C_{1}^{-1}\nu.$$
By means of energy inequality (the strong energy inequality \eqref{SEI} for $\tau=0$) $(\|u(0)\|_{L^{2}}
\geq\|u(t)\|_{L^{2}})$,  we obtain the uniform low
bound of $\|\Lambda^{\alpha}u(t)\|_{L^{2}}$. But this
contradicts with the energy inequality (the strong energy inequality \eqref{SEI} for $\tau=0$).

From the above discussion,
we could construct a
strong solutions $\tilde{u}$ with  the initial data
$u(T^{\ast})$. Together with  strong energy inequality \eqref{SEI} \[
      \|u(t)\|^{2}_{L^{2}}
       +2\nu\int^{t}_{T^{\ast}}\|\Lambda^{\alpha} u(s)\|^{2}_{L^{2}
      }\,ds\leq
       \|u(T^{\ast})\|^{2}_{L^{2}},  ~t\geq T^{\ast},                            \]
  and   Proposition \ref{the1.1},
we get  $\tilde{u}(t)=u(t)$ a.e. on $(T^{\ast},\infty)$.
This concludes the proof of Theorem \ref{eventual regularization}.
\end{proof}

\subsection{Hausdorff dimension estimate}
Based on the  eventual regularity of weak solutions, we know that the possible singular time points
of weak solutions to \eqref{GNS} must  be contained  in finite
interval $(0,T^{\ast})$. With the help of Proposition \ref{local} and
Proposition  \ref{the1.1},
 we have
 \begin{lemma} \label{1}
 The regular point set $\mathcal{R}$ can be
decomposed as follows \be\label{E2}\mathcal{R}=(\bigcup\limits_{i\in
A}(\tau_{i},\,s_{i})) \bigcup(T^{\ast},\infty),~s_{i}\in
\mathcal{IR},\ee where
$(\tau_{i},\,s_{i})\bigcap(\tau_{j},s_{j})=\emptyset$ for $i\neq j$
and the set  A is at most countable, where $T^*$ is same as in
\eqref{E1}.  Furthermore, the Lebesgue measure of the irregular
points set $\mathcal{IR}$ of the weak solution on time is zero.
\end{lemma}
 \begin{proof}
 Note that the weak solution $u$ belongs to
$L^{2}((0,\,T^{\ast});\,H^{\alpha})$  and satisfies strong energy inequality \eqref{SEI}
\[
      \|u(t)\|^{2}_{L^{2}}
+2\nu\int^{t}_{\tau}\|\Lambda^{\alpha} u(s)\|^{2}_{L^{2}         }\,ds\leq \|u(\tau)\|^{2}_{L^{2}},~\tau=0\, \text{or}\,a.e.\tau\geq0 ~\text{and}~  t\in
[0,T^{\ast}).
 \]
Therefore, for {\it a.e.}  $\tau\in (0,T^{\ast})$, the weak
solution $u(t)$ is a strong one on some interval
$(\tau,\tau+T(\tau))$ with the initial data $u(\tau)$ due to  Proposition \ref{local} and
Proposition  \ref{the1.1}.  The interval
$(\tau,\tau+T(\tau))$ can be extended to a maximal one $(\tau',s')$
 containing $(\tau,\tau+T(\tau))$ such that $s'\in \mathcal{IR}$.
Thence, $\mathcal{R}=(\bigcup\limits_{i\in A}(\tau_{i},\,s_{i}))
\bigcup(T^{\ast},\infty),$ where
$(\tau_{i},\,s_{i})\bigcap(\tau_{j},s_{j})=\emptyset$ for $i\neq j$
and the set  A is at most countable since the set composed by
mutually disjoint open intervals belonging to the line is finite or
countable. Claim \eqref{E2} is proved.\\
 Denote $I=\{s\in (0,T^{\ast})|u(s)\in H^{\alpha}\}$. It is clear that $$|(0,T^{\ast})\setminus I|=|I\setminus(\bigcup\limits_{i\in
A}(\tau_{i},\,s_{i}))|=0 .
$$
The Lebesgue measure of the $\mathcal{IR}$  is zero.
 \end{proof}
 It should be point out that the proof of the endpoint case $\alpha=5/6$ in Theorem \ref{Hausdorff} has been achieved based on the fact that the $1$ dimensional  Hausdorff measure coincides with Lebesgue measure on $\mathbb{R}$. Consequently, the following proof focuses its attention on the  regime $5/6<\alpha<5/4$.
 In order to  conclude the rest part of proof of Theorem \ref{Hausdorff}, we need exploit the necessary condition for $t_{0}$ to be a possible irregular point to generalized  Navier-Stokes equations \eqref{GNS} similar to  the Navier-Stokes equations
 $$
 \|\nabla u(t)\|_{L^{2}}\geq \frac{\nu^{3/4}C}{(t_{0}-t)^{1/4}},~ t<t_{0}
 $$
appearing  in Leray's  groundbreaking paper \cite{[Leray]}.

As mentioned  in Remark \ref{meaning}, the key inequality \eqref{grawall} help us to obtain
\begin{proposition}\label{coro}
Assume that $t_{0}$ is an irregular point of a weak solution $u$.
Then
\be\label{beha} \lim_{t\rightarrow
t_{0}}\|\Lambda^{\alpha}u(t)\|_{L^{2}} =\infty, \ee
Furthermore, for $5/6<\alpha<5/4$, there holds \be\label{blowup}
\|\Lambda^{\alpha}u(t)\|_{L^{2}}
\geq\f{\nu^{\f{5-2\alpha}{4\alpha}}C}{(t_{0}-t)^{\f{6\alpha-5}
{4\alpha}}},~ t<t_{0}. \ee
\end{proposition}
\begin{proof}
If  \eqref{beha} was not true, we can pick up $t_{k}$ such that
$t_{k}\rightarrow t_{0}$ as $k\rightarrow \infty$ with
$t_{k}<t_{0}$, moreover,
$$
\|\Lambda^{\alpha}u(t_{k})\|_{L^{2}}\leq C_{2}.
$$
Thanks to Proposition \ref{local},  Theorem \ref{the1.1} and the
strong energy inequality, proceeding as before,
the weak solution $u(t)$ can be seen as a strong solution
with the  initial data $u(t_{k})$ on $(t_{k},t_{k}+T(t_{k}))$, where
 $$
T(t_{k})=\f{\nu^{\f{5-2\alpha}{6\alpha-5}}}
{\f{2\alpha C}{6\alpha-5}\|\Lambda^{\alpha}u(t_{k})\|_{L^{2}}^{\f{4\alpha}{6\alpha-5}}
}\geq C_{2}^{-\f{4\alpha}{6\alpha-5}}C_{3}=T_{0},
 $$
 where $T_{0}$ does not   depend  on $k$.\\
We can choose $t_{k'}$ such that $t_{k'}+T_{0}>t_{0}$,
so the weak solution $u(t)$ is a strong solution on $(t_{k},t_{k}+T_{0})$,
 which is a contradiction to the fact that $t_{0} $ is an irregular point
  of the weak solution. Thus \eqref{beha} holds true.

Integrating the inequality \eqref{grawall} with respect with time
variable on $[t,\tau) $, we obtain
$$
\f{1}{\|\Lambda^{\alpha}u(t)\|_{L^{2}}^{\f{4\alpha}
{6\alpha-5}}}-\f{1}
{\|\Lambda^{\alpha}u(\tau)\|_{L^{2}}^{\f{4\alpha}{6\alpha-5}}
}\leq C\nu^{\f{2\alpha-5}{6\alpha-5}}(\tau-t),~t<\tau<t_{0}.
$$
Let $\tau\rightarrow t_{0}$ and \eqref{beha} yields  that
$$
\|\Lambda^{\alpha}u(t)\|_{L^{2}}\geq
\f{\nu^{\f{5-2\alpha}{4\alpha}}C}{(t_{0}-t)^{\f{6\alpha-5}
{4\alpha}}},
$$
which proves \eqref{blowup}.
\end{proof}
 Finally, we are ready to prove Theorem \ref{Hausdorff}.
\begin{proof}[Proof of Theorem \ref{Hausdorff}]
It follows from  the energy inequality (the strong energy inequality \eqref{SEI} for $\tau=0$) that
$$
\sum_{i\in A}\int^{s_{i}}_{\tau_{i}}
\|\Lambda^{\alpha}u(s)\|^{2}_{L^{2}}\,ds\leq C
\|u(0)\|^{2}_{L^{2}},$$ where  $A$ is defined as in
Lemma \ref{1}.\\
 With the  help of \eqref{blowup}, we infer that
$$
\sum_{i\in A}\int^{s_{i}}_{\tau_{i}}
\|\Lambda^{\alpha}u(s)\|^{2}_{L^{2}}\,ds\geq
C\sum_{i\in A}\int^{s_{i}}_{\tau_{i}}\f{1}{(s_{i}-s)^{\f{6\alpha-5}
{2\alpha}}}\,ds\geq
C\sum_{i\in A}\big(s_{i}-\tau_{i}\big)^{\f{5-4\alpha}
{2\alpha}}.
$$
Therefore, for any $\varepsilon>0$, there exists a
finite part $A_{1}$ of $A$ such that
$$
\sum_{i\in A\setminus A_{1}}\big(s_{i}-\tau_{i}\big)^{\f{5-4\alpha}
{2\alpha}}<\varepsilon.
$$
We denote the   finite interval $(0,T^{\ast})\setminus(\bigcup_{i\in A_{1}}(\tau_{i},\,s_{i}))=\bigcup\limits_{j=1}^{N} r_{j}$.\\
We note that
$$
\bigcup^{N}_{j=1}r_{j}=\mathcal{IR}\bigcup\Big(\bigcup_{i\in A\setminus
 A_{1}}\big(\tau_{i},\,s_{i}\big)\Big).
$$
Using
$|\mathcal{IR}|=0$, we find
$$
\text{diam} \{r_{j}\}=\sum_{\begin{subarray}{c}i\in A\setminus
 A_{1}\\(\tau_{i},\,s_{i})\subset r_{j}\end{subarray}}
(s_{i}-\tau_{i})\leq\sum_{i\in A\setminus  A_{1}}(s_{i}-\tau_{i}).
$$
Direct estimates give
$$\ba
\sum^{N}_{j=1}{\text{diam} \{r_{j}\}}^{\f{5-4\alpha}
{2\alpha}}&=\sum^{N}_{j=1} \Big(\sum_{\begin{subarray}{c}i\in
A\setminus A_{1}\\(\tau_{i},\,s_{i})\subset r_{j}\end{subarray}}
\big(s_{i}-\tau_{i}\big)\Big)^{\f{5-4\alpha}
{2\alpha}}\\
&\leq\sum^{N}_{j=1}\Big(\sum_{\begin{subarray}{c}i\in A\setminus
A_{1}\\(\tau_{i},\,s_{i})\subset r_{j}\end{subarray}}
\big(s_{i}-\tau_{i}\big)^{\f{5-4\alpha}
{2\alpha}}\Big)\\
&=\sum_{i\in A\setminus A_{1}} (s_{i}-\tau_{i})^{\f{5-4\alpha}
{2\alpha}}<\varepsilon. \ea$$ In view of   Definition \ref{HD}, the
proof is complete.
\end{proof}


\appendix
\section{ Existence of weak solutions}
\label{appendix}
\setcounter{section}{5}\setcounter{equation}{0}
To make our paper more self-contained and more readable, we outline the existence proof of weak solutions to \eqref{GNS}. The existence   of weak solutions of global weak solutions to the generalized MHD equations has been established by Wu in \cite{[Wu1]}. Contrary to Wu's work, we will modify some critical estimate. In addition, we show that the weak solutions satisfy the  strong energy inequality and  $L^{2}$ weak continuity.
\begin{theorem}\label{weak solution}
Let $u(0)$ be a divergence-free vector fields with finite energy. For any  $T>0$, there exists a weak solution to \eqref{GNS} with $\alpha>0$ in the following sense
\begin{enumerate}[(1)]
 \item$u\in L^{\infty}((0,\,T);\,L^{2})\cap L^{2}((0,\,T);\,H^{\alpha}).$
\item$u$~solves (\ref{GNS})  in the sense of distributions, namely,
   \be\label{dis1}
  \int_{0}^{T}\langle u,\,\phi_{t}\rangle-\nu\langle\Lambda^{\alpha}u,\,
  \Lambda^{\alpha}\phi\rangle
-\langle u\cdot\nabla u,\,\phi\rangle \,dt = -\langle u(0),\,\phi(0)\rangle,
  \ee
  for any $\phi\in C_{0}^{\infty}((0,\,T)\times \Omega)$.
 \item $u$ is $L^2$ weakly
continuous on the interval $[0,\,T)$, that is,
$$
\lim_{t\rightarrow t_{0}}\langle u(t)-u(t_{0}), \varphi \rangle=0,
~~\text{for all} \,~ t\in (0,\,T) \,~ \text{and all} ~\, \varphi\in L^{2}(\Omega)$$
and $L^2$
strongly continuous at time $0$, namely,
$$\lim_{t\rightarrow0}\|u(t)-u(0)\|_{L^{2}}=0.$$
\item $u$ verifies the strong   energy inequality
\be\label{SEI2} \|u(t)\|^{2}_{L^{2}}
+2\nu\int^{t}_{\tau}\|\Lambda^{\alpha} u(s)\|^{2}_{L^{2}
}\,dsu\leq \|u(\tau)\|^{2}_{L^{2}},~\ee
$\tau=0$ or $a.e. ~\tau>0$,  and~each   $t\in[\tau,T)$.
\end{enumerate}
\end{theorem}

\begin{proof}
We will apply the Galerkin method in the periodic domain similar to the classical Friedrich's method in the whole space to construct
 an approximate solution sequence. Let us define the operator $P_{N}$ by
$$
 P_{N}f(x)=\mathcal{F}^{-1}(1_{k\leq N}(k)\hat{u}(k))=\sum_{|k|\leq N}e^{ik\cdot x}\hat{u}(k)=\sum_{|k|\leq N}\int_{\Omega}e^{ik\cdot(x-x')}u(x')\,dx'.
$$
We seek approximate solution
$u^{N}=\sum\limits_{|k|\leq N}e^{ik\cdot x}\hat{u}(k)$
satisfy the following equations
\be\label{app1}
\partial_{t} u^{N}+\nu \Lambda^{2\alpha}u^{N}+P_{N}(u^{N}\cdot\nabla u^{N})+P_{N}\nabla p^{N}=0,~~~\text{div}\,u^{N}=0,\ee
together with the initial condition
$$u^{N}(x,0)=P_{N}u(x,0).$$
This system can be viewed as an ordinary differential equations on $L^{2}$. The Cauchy-Lipschitz theorem for ordinary differential system gives us the existence of a positive maximal time $T_{n}$ such that this system has a unique solution $u^{N}\in C([0,T_{n}];\,L^{2})$. According to the finite time blow-up theorem for ordinary differential system, it suffices to show the uniform energy estimates on $u^{N}$, which yields $T_{n}=T$ for an arbitrary but fixed $T>0$.

Multiplying \eqref{app1} by $u^{N}$ and integrating by parts, we see that
$$
\frac{1}{2}\frac{d}{dt}\|u^{N}(t)\|_{L^{2}}^{2}+\|\Lambda^{\alpha }u^{N}(t)\|_{L^{2}}^{2}=0,~~t\leq T,
$$
which implies that
\be\label{app2}
\frac{1}{2} \|u^{N}(t)\|_{L^{2}}^{2}+\int^{t}_{0}\|\Lambda^{\alpha }u^{N}(s)\|_{L^{2}}^{2}\,ds=\frac{1}{2} \|u^{N}(0)\|_{L^{2}}^{2}\leq\frac{1}{2} \|u(0)\|_{L^{2}}^{2},
\ee
This yields  that the $L^{2}$ norm of $u^{N}$ is controlled and $T_{n}=T$.
From the last inequality, by a diagonalization process, we can find a subsequence, again denoted by $u^{N}$, such that
$u^{N}\stackrel{\ast}{\rightharpoondown}u$ in $L^{\infty}((0,\,T);\,L^{2})$ and
$u^{N}\rightharpoondown u$ in $L^{2}((0,\,T);\,L^{2})$ as $N\rightarrow\infty$.

Just as the Navier-Stokes equations, one key point is to obtain  a strong convergence in $L^{2}((0,\,T);L^{2})$ to pass to the limit in the nonlinear term. The Aubin-Lions (\cite[Theorem 2.1, p.184]{[Teman2]}) Lemma allows us to achieve this. Thence, we turn to bound $\partial_{t}u^{N}$ via the equation
\be\label{eq1}
\partial_{t}u^{N}=\nu\Lambda^{2\alpha}u^{N}-P_{N}(u^{N}\cdot\nabla u^{N})-P_{N}\nabla p^{N}.
\ee
By   H\"older's inequality, integrating by parts and \eqref{app2}, for any $h\in L^{2}((0,\,T); H^{3})$, we find that
$$\ba
\int^{T}_{0}\langle\Lambda^{2\alpha}u^{N}(s), \,h(s) \rangle\, ds&\leq \int^{T}_{0}\|\Lambda^{\alpha}u^{N}(s)\|_{L^{2}}\|\Lambda^{\alpha}h(s) \|_{L^{2}}\,ds\\
&\leq \Big(\int^{T}_{0}\|\Lambda^{\alpha}u^{N}(s)\|_{L^{2}}^{2}\,ds\Big)^{1/2}
\Big(\int^{T}_{0}\|\Lambda^{\alpha}h(s)\|_{L^{2}}^{2}\,ds\Big)^{1/2}\\
&\leq \|u(0)\|_{L^{2}}\| h\|_{L^{2}((0,\,T);\,H^{3})}.
\ea$$
Applying H\"older's inequality and \eqref{app2}, we see that
\be\label{60}\ba
\int^{T}_{0}\langle P_{N}(u^{N}(s)\cdot\nabla u^{N}(s)),\,h(s) \rangle \,ds&\leq
\int^{T}_{0}\|u^{N}(s)\|_{L^{2}}\|u^{N}(s)\|_{L^{\frac{6}{3-2\alpha}}}
\|\nabla h(s)\|_{L^{\frac{3}{\alpha}}}\,ds\\
 &\leq C\|u^{N}\|_{L^{\infty}((0,\,T);\, L^{2})}
 \|\Lambda^{\alpha}u^{N}\|_{L^{2}((0,\,T);\,L^{2})}
\|h\|_{L^{2}(0,\,T;\,H^{3})} \\
 &\leq C\|u(0)\|_{L^{2}}\|h\|_{L^{2}((0,\,T);\,H^{3})},
\ea\ee
where we  used the   Sobolev embedding $H^{\alpha}\hookrightarrow L^{\frac{6}{3-2\alpha}}$ and
$ H^{2}\hookrightarrow L^{\frac{3}{\alpha}}$.\\
To bound $P_{N}\nabla p^{N}$, we
take divergence of \eqref{app1} to obtain the pressure equation
\be\label{pre}
\Delta P_{N}p^{N}=-\sum_{i,j}\partial_{i}\partial_{j}P_{N}(u_{i}^{N}u_{j}^{N}).
 \ee
In the light of  the classical Calder\'on-Zygmund Theorem, repeating the deduction process of \eqref{60}, we deduce that
$$\ba
\int^{T}_{0}\langle P_{N}\nabla p^{N}(s),h(s) \rangle \,ds &\leq
\int^{T}_{0}\|p^{N}(s)\|_{L^{\frac{3}{3-\alpha}}}
\|\nabla h(s)\|_{L^{\frac{3}{\alpha}}}\,ds\\
&\leq C\int^{T}_{0}|u^{N}(s)\|_{L^{2}}\|u^{N}(s)\|_{L^{\frac{6}{3-2\alpha}}}
\|\nabla h(s)\|_{L^{\frac{3}{\alpha}}}\,ds\\
&\leq
 C\|u(0)\|_{L^{2}}\|h\|_{L^{2}((0,\,T);\,H^{3})}.
\ea$$
Combining these estimates, we obtain
\be\label{al}
\partial_{t}u^{N}\in L^{2}((0,\,T);\,H^{-3}),\ee
which together with \eqref{app2} yields that
$$u^{N}\rightarrow u ~~~\text{in} ~~~L^{2}((0,\,T);L^{2}).$$
Taking the scalar product of \eqref{app1} with $\phi\in C_{0}^{\infty}((0,\,T)\times\Omega)$ and integrating, we have
$$\int_{0}^{T}\langle u^{N},\,\phi_{s}\rangle-
\nu\langle\Lambda^{\alpha}u^{N},\,\Lambda^{\alpha}\phi\rangle
-\langle u^{N}\cdot\nabla u^{N},\,\phi\rangle \,ds = -\langle u^{N}(0),\,\phi(0)\rangle.$$
Now we can pass to the limit in the above equations exactly as the standard Naveri-Stokes equations (for example \cite[p.196]{[Teman2]}) to show \eqref{dis1}.

In order to prove property (3),
we shall apply Arzel\`a-Ascoli theorem to pass to the limit
\be\label{app3}
\lim_{N\rightarrow\infty} \langle u^{N}(t),\varphi\rangle=\langle u (t),\varphi\rangle, ~~\text{for each}~~ t\in[0,T], ~\varphi\in L^{2}(\Omega),
\ee
hence,
it is enough to show that the function
$$
f_{N}(t)=\langle u^{N}(t),\varphi\rangle
$$
is uniformly bounded and equicontinuous  on $[0,T]$.\\
It follows from  \eqref{app2} that $f_{N}(t)$ is uniformly bounded.
Notice that $\varphi$  is independent of time,
arguing in the same manner as \eqref{al},
$$\ba
|f_{N}(t_{1})-f_{N}(t_{2})| &\leq \Big|\int^{t_{1}}_{t_{2}}\langle\partial_{s}u^{N}(s),\, \varphi\rangle \,ds\Big|\\
&\leq  C |t_{1}-t_{2}|^{1/2}\|\partial_{s}u^{N}\|_{L^{2}(0,T;\,H^{-3})}
\|\varphi\|_{H^{3}}\\
&\leq  C |t_{1}-t_{2}|^{1/2}\|u(0)\|_{L^{2}}\|\varphi\|_{H^{3}},
\ea$$
which implies that $f_{N}(t)$ equicontinuous when $\varphi\in H^{3}$.
In the light of Arzel\`a-Ascoli theorem, we infer that
$$
\lim_{N\rightarrow\infty} \langle u^{N}(t), \varphi\rangle=\langle u (t), \varphi\rangle, ~~\text{for each}~~ t\in[0,\,T],
$$
and  $f(t)=\langle u(t),\varphi\rangle$ is  continuous function where $\varphi\in H^{3}$.
With the help of the fact that  $H^{3}$ is dense
 in $ L^{2}$, by dense argument,  we  conclude that the weak solutions is  $L^{2}$ weakly continuous.

Making use of \eqref{app3}, we see that
\be\label{last1}\|u(t)\|_{L^{2}}=\sup_{\|v\|_{L^{2}}\leq1}\langle u(t),\,v\rangle=\sup_{\|v\|_{L^{2}}\leq1}\liminf_{N\rightarrow\infty}\langle u(t),\,v\rangle
\leq \liminf_{N\rightarrow\infty}\|u^{N}(t)\|_{L^{2}}, \ee
 for any $ t\in[0,T]$, which together with \eqref{app2} implies
$$\lim_{t\rightarrow0}\|u(t)\|_{L^{2}}\leq \|u(0)\|_{L^{2}}.$$
Similar  to \eqref{last1}, utilize the $L^{2}$ weak continuity to obtain
$$\|u(0)\|_{L^{2}}=\sup_{\|v\|_{L^{2}}\leq1}\langle u(0),\,v\rangle=
\sup_{\|v\|_{L^{2}}\leq1}\liminf_{t\rightarrow0}\langle u(t),\,v\rangle
\leq \liminf_{t\rightarrow0} \|u(t)\|_{L^{2}}.$$
Recalling that  parallelogram identity  holds in Hilbert space $L^{2}$, we know that $\lim\limits_{t\rightarrow0}\|u(t)\|_{L^{2}}=\|u(0)\|_{L^{2}}$  and $L^{2}$ weak continuity mean that $\lim\limits_{t\rightarrow0}\|u(t)-u(0)\|_{L^{2}}=0.$

At last, we check that the energy inequality corresponding to $s=0$ in \eqref{SEI2} and the strong energy inequality \eqref{SEI2} are valid.
On the one hand,
it  follows from \eqref{app2}, \eqref{last1} and  Fatou's Lemma that
$$ \|u(t)\|^{2}_{L^{2}}
+2\nu\int^{t}_{0}\|\Lambda^{\alpha} u(s)\|^{2}_{L^{2}}\,ds\leq \|u(0)\|^{2}_{L^{2}},~t\in
(0,\,T). $$
On the other hand, there exists a set $E\subset [0,T)$ whose Lebesgue measure is zero  and a subsequence (still denoted by the same symbol)
$u^{N}$
such that
$$
u^{N}(\tau)\rightarrow u(\tau) ~\text{as} ~N\rightarrow\infty ~~\text{in}~~ L^{2},~
\text{for each } \tau\in[0,\,T)/E
$$
due to the fact that $u^{N}(t)\in L^{2}([0,\,T);\,H^\alpha)$ and
$H^\alpha\hookrightarrow\hookrightarrow L^{2}$.
Based on this, we can pass to the limit of
$$
\frac{1}{2} \|u^{N}(t)\|_{L^{2}}^{2}+\nu\int^{t}_{\tau}\|\Lambda^{\alpha }u^{N}(s)\|_{L^{2}}^{2}\, ds=\frac{1}{2} \|u^{N}(\tau)\|_{L^{2}}^{2}, ~
$$
 to arrive at
 $$
\frac{1}{2} \|u(t)\|_{L^{2}}^{2}+\nu\int^{t}_{\tau}\|\Lambda^{\alpha }u(s)\|_{L^{2}}^{2}\,ds\leq\frac{1}{2} \|u(\tau)\|_{L^{2}}^{2},
$$
for almost every $ \tau\in[0,T)$ and each $t\in[\tau,\,T)$.

\end{proof}

\noindent
{\bf Acknowledgements:}
The authors would like to express their sincere gratitude to Gang Wu and
Jiefeng Zhao for many helpful discussions on this topic. The first author is supported by National
Natural Sciences Foundation of China (No. 11171229, No.11231006) and Project of Beijing Chang Cheng Xue Zhe.


\end{document}